\documentclass[a4paper, 12pt]{extarticle}
\usepackage{float}
\usepackage{graphicx}
\usepackage{amscd}
\usepackage[english]{babel}
\usepackage{latexsym,amsfonts,amssymb,amsmath,longtable,amsthm}
\usepackage[pic]{xy}
\usepackage{amsmath}
\newtheorem{lem}{{\scshape Lemma}}

\newtheorem{example}{{\scshape Example}}
\newtheorem{ttt}{{\scshape Theorem}}
\newtheorem{prp}{{\scshape Proposition}}

\begin{document}
\title{Virtual quandle for links in lens spaces}
\author{A.~Cattabriga\footnote{The author is supported by GNSAGA of INdAM and University of Bologna, funds for selected research topics.},~~~ T.~Nasybullov\footnote{The author is supported by the Research Foundation -- Flanders (FWO), app. 12G0317N.}}
\date{}
\maketitle
\begin{abstract}
We construct a virtual quandle for links in lens spaces $L(p,q)$, with $q=1$.  This invariant has two valuable advantages over an ordinary fundamental quandle for links in lens spaces: the virtual quandle is an essential invariant and the presentation of the virtual quandle can be easily written from the band diagram of a link.\\


\end{abstract}

\section{Introduction}
Over the years knot theory has worked with knots and links in the three-dimensional sphere $S^3$. However, together with improving knowledge about 3-manifolds, great attention has been paid to knots and links  in manifolds different from $S^3$: Seifert fiber spaces \cite{gabman} (in particular, lens spaces \cite{catmanmul,catmanrig,dro,hospr,manf}), thickened surfaces \cite{ilyman, kau}, arbitrary 3-manifolds \cite{fenrou, lamrou}. In this paper we focus on knots and links in lens spaces $L(p,q)$ with $q=1$. These manifolds are the simplest ones (beside $S^3$) where we can study knots and links since every link in $L(p,1)$ can be presented by integer $p$-surgery over the unknot (in cases of other manifolds the surgery link can be much more difficult).
 Moreover, there are interesting articles explaining applications of knots in lens spaces to other fields of science: \cite{ste} exploits them to describe topological string theories and
\cite{bucmau} uses them to describe the resolution of a biological DNA recombination
problem.

A lot of link invariants can be generalized from links in $S^3$ to links in $L(p,q)$. Kauffman bracket skein module \cite{hospr}, knot Floer homology \cite{bakgrihed}, HOMFLY-PT polynomial \cite{cor} all have analogues in lens spaces.

Despite the fact that some invariants can be generalized to links in lens spaces, sometimes it is very difficult to use them.
For example, the fundamental quandle of a link  which has a very simple topological description for links in $S^3$ can be easily generalized to links in $L(p,q)$. However, an explicit procedure to write down a  presentation of a fundamental quandle for links in $L(p,q)$ by generators and relations directly from a (band, disc, grid) diagram is known only in the case $(p,q)=(2,1)$ (see \cite{gor}), so it is almost impossible to use this invariant. Another disadvantage of the ordinary fundamental quandle for links in $L(p,q)$ is that the
fundamental quandle of the link $K\subset L(p,q)$ is isomorphic to the fundamental quandle of its lift $\pi^{-1}(K)\subset S^3$, where $\pi:S^3\to L(p,q)$ denotes the universal covering map. So we cannot distinguish links  with equivalent liftings. Such invariants are called \emph{inessential} \cite{catmanrig}. In this paper we define  the notion of  virtual quandle for links in $L(p,1)$. Virtual quandles were firstly introduced in \cite{man,man2} as generalizations of quandles for virtual links. Even if in the present paper we do not work with virtual links, we use the name ``virtual quandle'' since our invariant has the same algebraic structure.

The paper is organized as follows. In Section \ref{band} we review the definition of  band diagrams for links in lens spaces, while in Section \ref{quandle} we recall those of quandle and virtual quandle, proving some results that will be useful in the following.  In Section \ref{s33} we construct a virtual quandle  for links in $L(p,1)$ and  prove that it is an invariant (see Theorem \ref{tt2}). Finally, in Section \ref{exex} we investigate  some properties of this invariant, proving  that it is   essential  (see Proposition \ref{esss}).

\section{Band diagrams for links in lens spaces}\label{band}
In this section we recall the notion of  band diagrams for links in lens spaces (which were firstly introduced in \cite{hospr}) and recall how  Reidemeister moves for such band diagrams look like. 
   
We start by giving  the definition of Dehn surgery in $S^3$. Let $K$ be a knot in $S^3$, denote by $N(K)$  a closed tubular neighborhood of $K$ and let $\gamma$ be a simple closed curve on $\partial N(K)$. \textit{The Dehn surgery on $K$ along $\gamma$} is the manifold $M$ obtained by gluing  $S^3\setminus\textup{int}(N(K))$ with  a solid torus $D^2\times S^1$  along their boundaries  via a homeomorphism which identifies $\gamma$ with the boundary of a   meridian disc. The curve $\gamma$ is called the \textit{slope} of the surgery and the homeomorphism type  of the resulting manifold depends only on the homology class of $\gamma$ in $\partial N(K)$, up to orientation change. We fix a base $(m,l)$ for $H_1(\partial N(K))$, such that $m$ is a meridian of $K$  and  $l$ has  algebraic intersection 1
with $m$. If the homology class of $\gamma$ is $pm+ql$, with $p,q\in \mathbb Z$, we call the surgery \textit{rational} and we say that it has  
\emph{framing index} $p/q$.

Every link  $L$ in a manifold $M$ obtained via a $p/q$-surgery over  $K$, can be presented by a diagram of a link $L' \cup K$ in $S^3$, where the knot $K$ is equipped with a  number $p/q\in\mathbb Q\cup\{\infty\}$,  the surgery along $K$ gives $M$ and $L$ is the image of $L'$ under the surgery operation. This presentation of $L$ is called \emph{a mixed link diagram} and the knot $K$ is called the \emph{surgery knot}. In order to simplify notation we use the same symbol $L$ to denote both $L'$ and $L$. 

The \textit{lens space} $L(p,q)$ is the 3-manifold obtained by  a rational surgery on the unknot $U$ in $S^3$ with framing index $-p/q$. So, 
a link $L$ in $L(p,q)$ can be described by a mixed link diagram of the link $L \cup U$ (see, for example, Figure \ref{mixxx}).  

\begin{figure}[bh]
\noindent\centering{
\includegraphics[width=45mm]{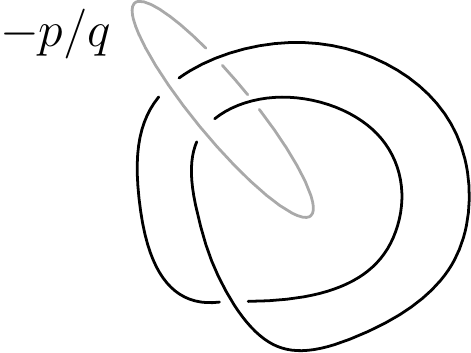}}
\caption{Example of a mixed link diagram of a link in $L(p,q)$.}
\label{mixxx}
\end{figure}

In order to construct a band diagram for a link $L$ in $L(p,q)$ we present $S^3$ as a one-point compactification of $\mathbb{R}^3$ and fix coordinates $(x_1,x_2,x_3)$ in $\mathbb{R}^3$ such that the surgery knot (which is the unknot in $S^3$) is described as the $x_3$ axis in $\mathbb{R}^3$. A regular orthogonal projection of a mixed link diagram onto the plane $x_1x_2$ is called \emph{a punctured disc diagram} of $L$ (see Figure \ref{punban} on the left). In a punctured disc diagram the surgery knot is depicted as  a puncture. A \emph{band diagram} for a link $L$ in $L(p,q)$ can be obtained from a punctured disk diagram by (i)
cutting a disc containing the punctured disk diagram and centered in the pucture along
a half-line starting from the puncture and avoiding the crossings of $L$ and (ii) deforming the resulting 
annulus into a rectangle (see Figure \ref{punban} on the right). Clearly, an orientation of the link induces an orientation on the resulting  band diagram. 

\begin{figure}[bht!]
\noindent\centering{
\includegraphics[width=110mm]{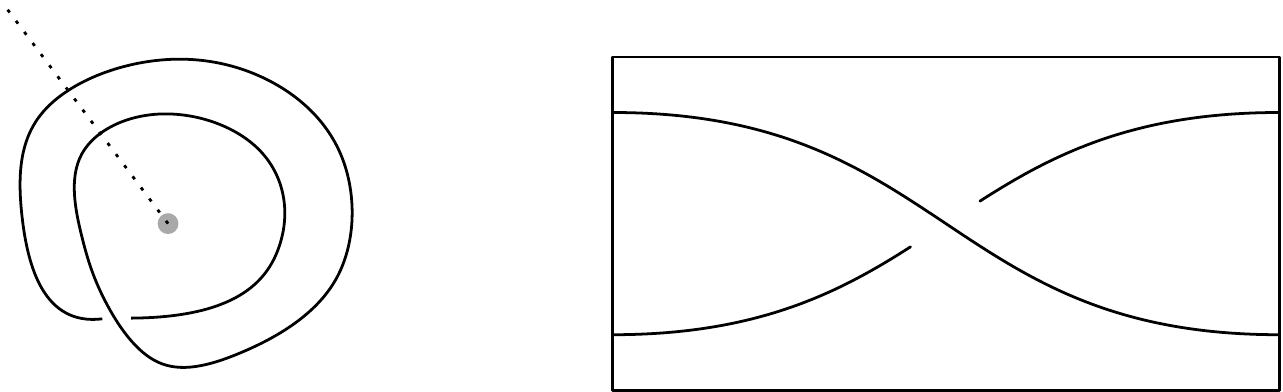}
\caption{A punctured disc diagram (on the left) and the  corresponding band diagram (on the right).}
}
\label{punban}
\end{figure}

The intersection points of the link with the left (resp. right) side of the band diagram are called \emph{left (resp. right) boundary points}. Left and right boundary points together are called \emph{boundary points} of a band diagram. 

Two band diagrams are equivalent under the three classical Reidemeister moves $R_1$, $R_2$, $R_3$ inside the rectangle and a global move called  {\it the $SL$-move}  (see \cite{lamrou}). In the case of lens spaces $L(p,1)$ the $SL$-move is depicted in Figure \ref{SL}.

\begin{figure}[bh]
\noindent\centering{
\includegraphics[width=55mm]{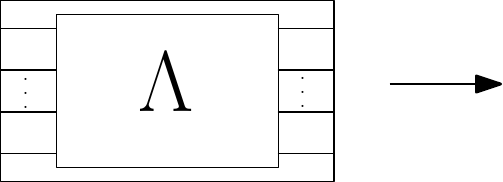}\includegraphics[width=80mm]{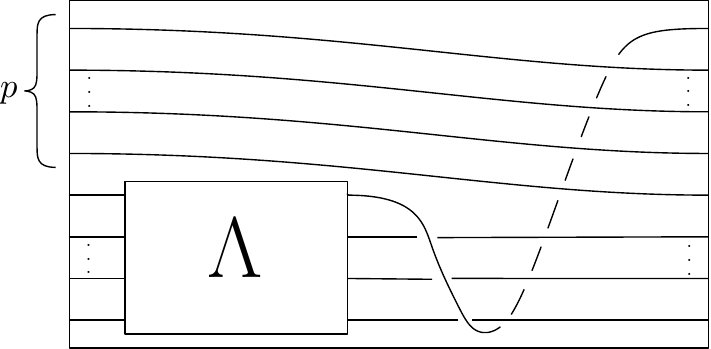}}
\caption{The $SL$-move for links  in $L(p,1)$.}
\label{SL}
\end{figure}

\section{Quandles and virtual quandles}\label{quandle}
In this section we review some facts about quandles and virtual quandles from the algebraic point of view.

\subsection{Quandles}
A \emph{ quandle} is a set $Q$ endowed with a binary algebrabraic operation $(x,y)\mapsto x^y$ satisfying  the following axioms:
\begin{enumerate}
\item for all $x\in Q$ $x^x=x$;
\item for all $x\in Q$ the map $i_x:y\mapsto y^x$ is a bijection on $Q$;
\item for all $x,y,z\in Q$ $x^{yz}=x^{zy^z}$.
\end{enumerate}
Here for simplicity we use symbols $x^{yz}$ and $x^{y^z}$ to denote $(x^y)^z$ and $x^{(y^z)}$ respectively.

The notion of a quandle was first introduced by Joyce \cite{joy} and, independently, Matveev \cite{mat} as an invariant for links in the $3$-sphere. After that quandles were also studied from an algebraic point of view (see, for example, \cite{car}).

\begin{example}{\rm For a group $G$  and a number $k\in\mathbb{N}$ denote by $Conj_k~G$ the quandle having the same underlying set as $G$ and with operation $x^{y}=y^{-k}xy^k$. This quandle is called the  \emph{$k$-th conjugacy quandle.}}
\end{example}

A bijection $\varphi:Q\to Q$ is called an \textit{automorphism} of a quandle $Q$ if  for all $x,y \in Q$ the equality $\varphi(x^y)=\varphi(x)^{\varphi(y)}$ holds. In particular, from the second and the third axioms follows that the map $i_x$ is an automorphism of $Q$ for each $x\in Q$. The group generated by $i_x$ for all $x\in Q$ is called \emph{the group of inner automorphisms of $Q$} and is denoted by ${\rm Inn}(Q)$.
We denote by $y^{x^{-1}}$ the image of the element $y$ under the map $i_x^{-1}$: this notation should not suggest that $x^{-1}$ is an element of $Q$, however it is practically convenient since $y^{xx^{-1}}=y^{x^{-1}x}=y$ for all $x,y\in Q$.  From the third axiom of quandles for all $x,y\in Q$ we have $x^{yz}=x^{zy^z}$, and replacing $x$ by $x^{z^{-1}}$, we obtain the equality
\begin{equation}\label{nomore}
x^{z^{-1}yz}=x^{y^z}.
\end{equation}
If an element of a quandle $Q$ has the form $x^w$ for $x\in Q$ and $w\in F(Q)$, then we say that the element $x$ is on the {\it  primary level} and the element $w$ is on the {\it  operator level}. Equality  (\ref{nomore}) says that we do not need more than these two levels and so every element of a quandle has the form  $x^w$. Moreover,  if $x^v=x^w$ for all $x\in Q$ then we use the notation $v\equiv w$ and call this expression  \emph{the operator relation} (see \cite{fenrou}). 

\subsection{Virtual quandles}
A \emph{virtual quandle} is a set $VQ$ endowed with a binary algebraic operation $(x,y)\mapsto x^y$ and a unary algebraic operation $x\mapsto f(x)$  satisfying  the following axioms:
\begin{enumerate}
\item $VQ$ with the operation $(x,y)\mapsto x^y$ is a quandle;
\item $f$ is a bijection of  $VQ$ such that  $f(x^y)=f(x)^{f(y)}$.
\end{enumerate}

\begin{example}{\rm Let $G$ be a group, $k\in\mathbb{N}$ and $g$ be a fixed element of $G$. If we add to the quandle $Conj_k~G$ the operation $i_g$, then we obtain a virtual quandle.}
\end{example}

A bijection $\varphi:VQ\to VQ$ is an \textit{automorphism} of the virtual quandle $VQ$ if for all $x,y\in VQ$ we have $\varphi(x^y)=\varphi(x)^{\varphi(y)}$ and $\varphi(f(x))=f(\varphi(x))$. The selfmap $f$ is clearly an automorphism of $VQ$, moreover $f$ belongs to the center of the group ${\rm Aut}(VQ)$. On the contrary the bijection $i_x$ is not necessarily an automorphism of a virtual quandle.
Indeed if $i_x$ is an automorphism, then for all $y\in VQ$ we have $i_x(f(y))=f(i_x(y))$ and $f(y)^x=f(y)^{f(x)}$. So, $i_x$ is an automorphism of $VQ$ if and only if $x\equiv f(x)$, for all $x\in VQ$.

If $VQ$ is a virtual quandle, then we denote by $VQ^{-}$ the quandle obtained from $VQ$ just forgetting about the operation $f$. Every automorphism of $VQ$ induces an automorphism of $VQ^{-}$ and every map $i_x$ induces an automorphism of $VQ^{-}$.

As for groups, it is possible  \textit{to present (virtual) quandle by generators and relations} as follows.  
Given  a set of letters $X=\{x_1,x_2,\dots\}$ denote by $A(X)$ the inductively defined set of words obtained from $X$ by all possible subsequent application of the operations $(x,y)\mapsto x^y$, $(x,y)\mapsto x^{y^{-1}}$, $x\mapsto f(x)$, $x\mapsto f^{-1}(x)$. Let $\{r_i,s_i~|~i\in I\}$ be some set of words from $A(X)$ and $R=\{r_i=s_i~|~i\in I\}$ be a set of formal equalities. Denote by $\sim$ an equivalence relation on $A(X)$ which consists of the following relations:
\begin{enumerate}
\item  for all $x,y\in A(x)$ $x\sim \left(x^y\right)^{y^{-1}}\sim \left(x^{y^{-1}}\right)^y$;
\item  for all $x\in A(X)$ $x\sim f^{-1}\left(f(x)\right)\sim f\left(f^{-1}(x)\right)$;
\item  for all $x\in A(X)$ $x\sim x^x\sim x^{x^{-1}}$;
\item $\left(x^y\right)^z\sim\left(x^z\right)^{y^z}$ for all $x,y,z\in A(X)$;
\item for all $x,y\in A(X)$ $f\left(x^{y^{-1}}\right)\sim f(x)^{f(y)^{-1}}$, $f^{-1}\left(x^{y^{-1}}\right)\sim f^{-1}(x)^{f^{-1}(y)^{-1}}$,\\
  $f(x^y)\sim f(x)^{f(y)}$, $f^{-1}(x^y)\sim f^{-1}(x)^{f^{-1}(y)}$;
\item $r_i\sim s_i$ if the equality $r_i=s_i$ belongs to $R$.
\end{enumerate}
The quotient $A(X)/_{\sim}$ with well defined operations $(x,y)\mapsto x^y$ and  $x\mapsto f(x)$ is obviously a virtual quandle. We say that this virtual quandle is \emph{presented by the set of generators $X$ and the set of relations $R$} and denote it by $\langle X~|~R\rangle$. Every relation in the virtual quandle $VQ$ has the form $x^v=y^w$, where $x,y$ are elements of $VQ$ and $v,w$ belong to the free group $F(VQ)$ generated by $VQ$. 


\subsection{$n$-splitting automorphisms of quandles}
\label{splittt}

Let $f$ be an automorphism of a quandle $Q$. In general we cannot write the equalities $f(x^{y^{-1}})=f(x)^{f(y^{-1})}$ and $f(x^{yz})=f(x)^{f(yz)}$ since the expressions $f(y^{-1})$ and $f(yz)$ are not defined. However it is easy to show that the following equalities hold $$f(x^{y^{-1}})=f(x)^{f(y)^{-1}}~~~~~~~~~f(x^{yz})=f(x)^{f(y)f(z)}.$$

Moreover, an automorphism $f$ of a quandle $Q$ induces an automorphism $f_*$ of ${\rm Inn}(Q)$ which acts on the generators of ${\rm Inn}(Q)$ by the rule $f_*(i_x)=i_{f(x)}$.
As a consequence, on the operator level  we can write expressions of type $f(yz)$ and $f(y^{-1})$ meaning that $x^{f(yz)}=f_*(i_zi_y)(x)$ and $x^{f(y^{-1})}=f_*(i_y^{-1})(x)$. Moreover if $x$ belongs to the free group $F(Q)$ generated by the elements of $Q$ then on the operator level we can write expressions of type $f(x)$.

An automorphism $f$ of a quandle $Q$ is said to be an \emph{$n$-splitting automorphism} of $Q$ if for every element $x\in {\rm Inn}(Q)$ the equality $xf_*(x)\dots f_*^{n-1}(x)=1$ holds. In other words, $f$ is an $n$-splitting automorphism of $Q$ if $xf(x)\dots f^{n-1}(x)\equiv1$ for all $x\in F(Q)$.

 \begin{example}{\rm Let $G$ be a group, $N=\langle x^n~|~x\in G\rangle$ be a subgroup of $G$ generated by $n$-th powers of elements from $G$ and $g\in G/N$ be a fixed element of the quotient. Then $i_g$ is an $n$-splitting automorphism of the quandle $Conj_1~G/N$.}
 \end{example}

 We end the section proving some useful properties of $n$-splitting automorphisms.

\begin{lem}\label{Lem1}  Let $f\in{\rm Aut}(Q)$ be an $n$-splitting automorphism of a quandle $Q$. The following properties hold:

\begin{enumerate}
 \item[\textup{1)}] $f_*^n(x)= x$ for all $x\in {\rm Inn}(Q)$;
 \item[\textup{2)}]  $f_*^{n-1}(x)f_*^{n-2}(x)\dots x=1$ for all  $x\in {\rm Inn}(Q)$;
 \item[\textup{3)}] for every $z\in Q$ the automorphism $g=i_zf$ is an $n$-splitting automorphism of $Q$. Moreover if $f^n=id$, then $g^n=id$.
\end{enumerate}
\end{lem}
\begin{proof}
1) From the $n$-splitting relation we have
$$1= f_*(x)f_*^2(x)\dots f_*^{n}(x)= x^{-1}xf_*(x)f_*^2(x)\dots f_*^{n}(x)= x^{-1}f_*^n(x).$$
Therefore $f_*^n(x)= x$.

2) follows from 1) and the $n$-splitting relation.

3) Let $y$ be an arbitrary element of $F(Q)$, then the following sequence of equalities holds
\begin{align}
\notag yg(y)\dots g^{n-1}(y)&\equiv yf(y)^zf^2(y)^{f(z)z}\dots f^{n-1}(y)^{f^{n-2}(z)\dots z}\\
\notag&\equiv yz^{-1}f(y)zz^{-1}f(z)^{-1}f^2(y)f(z)z\dots f^{n-1}(y)f^{n-2}(z)\dots z\\
\notag&\equiv yz^{-1}f(yz^{-1})\dots f^{n-1}(yz^{-1})f^{n-1}(z)f^{n-2}(z)\dots z\equiv 1.
\end{align}
In the last equality we used item 2) of the lemma. So $g$ is an $n$-splitting automorphism of $Q$. For every element $x\in Q$ we have
$g^n(x)=f^n(x)^{f^{n-1}(z)f^{n-2}(z)\dots z}=f^n(x)$,
therefore if $f^n=id$, then $g^n=id$. 
\end{proof}

Note that equality 1) of  Lemma \ref{Lem1} does not imply the equality $f^n(x)=x$. It just says that $f^n(x)\equiv x$ for all $x\in F(Q)$.

\section{The virtual quandle for links in $L(p,1)$}\label{s33}
In this section we define the notion of the virtual quandle $VQ(K)$ for a link $K$ in a lens space $L(p,1)$ and prove that it is invariant under all moves $R_1$, $R_2$, $R_3$, $SL$. 

An \emph{arc} of a band diagram of a link $K$ in $L(p,1)$ is a connected  part of the diagram  bounded by two overpasses, or two boundary points, or one overpass and one boundary point. Referring to Figure \ref{Ban}, we label (i) the arcs which are bounded by (at least) one left boundary point by the letters $x_1,x_2,\dots,x_n$, (ii) the arcs which are bounded by (at least) one right boundary point  by the letters $y_1,y_2,\dots,y_n$ and (iii) all the remaining arcs by the letters $z_1, z_2,\dots,z_k$.  Note that some of $x_1,\dots,x_n,y_1,\dots y_n, z_1,\dots,z_k$ may be identified. Choose an orientation of the link and for each right boundary point $y_i$ define the number $\varepsilon_i$ by the following rule: if the arrow $y_i$ runs from the left to the right, then $\varepsilon_i=1$, otherwise $\varepsilon_i=-1$.

\begin{figure}[bht!]
\noindent\centering{
\includegraphics[width=50mm]{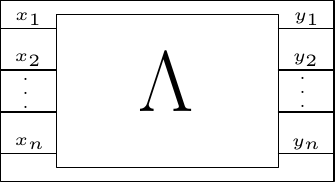}}
\caption{Labelling arcs of a band diagram.}
\label{Ban}
\end{figure}

We define the \emph{virtual quandle} $VQ(K)$ of a link $K\subset L(p,1)$ as the virtual quandle having the presentation $VQ(K)=\langle X~|~I\cup B\cup S\rangle$, where:

\begin{itemize}
\item \emph{the set of generators $X$} is $\{x_1,\ldots,x_n,y_1,\ldots, y_n, z_1,\ldots,z_k\}$; 

\item \emph{the set of inner relations $I$} consists of relations which can be written from the part $\Lambda$ of the band diagram contained in the interior of the rectangle. Here we have (i) relations of identifications between some of $x_1,\dots,x_n$, $y_1,\dots y_n$, $z_1,\dots,z_k$ (if it is necessary) and (ii)  as in the classical case, for each crossing,  the  relation  $x^y=z$,  where $x$, $y$, $z$ are the  labels of the arcs involved in the crossing as depicted in   Figure \ref{crosrule};

\begin{figure}[hbt!]
\noindent\centering{
\includegraphics[width=20mm]{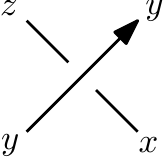}}
\caption{Labels of the arcs in classical crossings.}
\label{crosrule}
\end{figure}
\item \emph{the set of boundary relations $B$} consists of the $n$ relations $f(x_i)=y_i$ for $i=1,\ldots, n$ and one operator relation $y_n^{\varepsilon_n}y_{n-1}^{\varepsilon_{n-1}}\dots y_1^{\varepsilon_1}\equiv1$; 

\item \emph{the set of splitting relations $S$} consists of relations which state that  $f$ is a $p$-splitting automorphism of $VQ(K)^-$ and has order $p$:
$xf(x)\dots f^{p-1}(x)\equiv1$ for all $x\in F(X)$ and $f^p(x)=x$ for all $x\in VQ(K)$.

\end{itemize}

\begin{example}\label{xxx} 
{\rm Let $K$ be the knot in $L(p,1)$, whose band diagram is depicted in Figure~\ref{exx4}. 
\begin{figure}[bht!]
\noindent\centering{
\includegraphics[width=50mm]{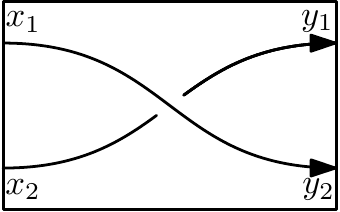}}
\caption{Example of a band diagram for the knot $K$ in $L(p,1)$.}
\label{exx4}
\end{figure}

We have four generators $x_1$, $x_2$, $y_1$, $y_2$. Since both $y_1$, $y_2$ run from the left to the right, $\varepsilon_1=\varepsilon_2=1$. Inner relations are $x_1=y_2$, $x_2^{x_1}=y_1$. Boundary relations are $f(x_1)=y_1$, $f(x_2)=y_2$, $y_2y_1\equiv1$. Splitting relations are $xf(x)\dots f^{p-1}(x)\equiv1$, $f^p(x)=x$ for all $x$. Therefore $VQ(K)$ has the following presentation
\begin{align}
\notag VQ(K)=\langle x_1,x_2,y_1,y_2~|~&x_1=y_2,x_2^{x_1}=y_1,f(x_1)=y_1,f(x_2)=y_2,y_2y_1\equiv1,\\
\notag&\forall x~~xf(x)\dots f^{p-1}(x)\equiv1,\forall x~~f^p(x)=x \rangle.
\end{align}
Using the relations $x_1=y_2$, $x_2^{x_1}=y_1$ we can delete the elements $y_1,y_2$ from the set of generators. Then $VQ(K)$ can be rewritten as
\begin{align}
\notag VQ(K)=\langle x_1, x_2~|~&f(x_1)=x_2^{x_1}, f(x_2)=x_1, x_2x_1\equiv1\\
\notag &\forall x~~xf(x)\dots f^{p-1}(x)\equiv1,\forall x~~f^p(x)=x\rangle.
\end{align}
Using the relation $f(x_2)=x_1$ we can delete $x_1$ from the set of generators. Then if we replace the letter $x_2$ just by $x$ we have
$$VQ(K)=\langle x~|~f^2(x)=x^{f(x)}, xf(x)\equiv1,\forall x~~xf(x)\dots f^{p-1}(x)\equiv1,\forall x~~f^p(x)=x\rangle.
$$
From the relations $f^2(x)=x^{f(x)}$ and $xf(x)\equiv1$ we obtain the relation $f^2(x)=x$. If $p$ is odd, then from the relations $f^2(x)=x$ and $f^p(x)=x$ we get $f(x)=x$ and the virtual quandle has the following presentation
$$VQ(K)=\langle x~|~f(x)=x\rangle.$$

If $p$ is even, then the relation $f^p(x)=x$ is a consequence of the relation $f^2(x)=x$, and the relation $xf(x)\dots f^{p-1}(x)\equiv1$ can be obtained from the relation $xf(x)\equiv1$. Then the virtual quandle has the following presentation.
$$VQ(K)=\langle x~|~f^2(x)=x, xf(x)\equiv1\rangle$$
}
\end{example}

\begin{example} {\rm Let $U$ be the unknot in $L(p,1)$. A band diagram for $U$ is depicted in Figure~\ref{exx45}. 
\begin{figure}[H]
\noindent\centering{
\includegraphics[width=50mm]{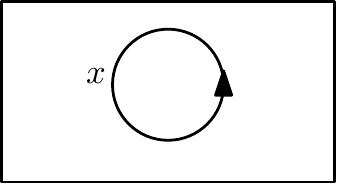}
}
\caption{Band diagram for the  unknot in $L(p,1)$.}
\label{exx45}
\end{figure}
Here we have only one generator $x$ and only splitting relations, so 
$$VQ(U)=\langle x~|~\forall x~~xf(x)\dots f^{p-1}(x)\equiv1,\forall x~~f^p(x)=x\rangle.$$
}
\end{example}

\begin{example}\label{xx} 
{\rm Let $L$ be the $n$-component link  in $L(p,1)$, whose band diagram is depicted in Figure~\ref{exx6}. 

\begin{figure}[H]
\noindent\centering{
\includegraphics[width=50mm]{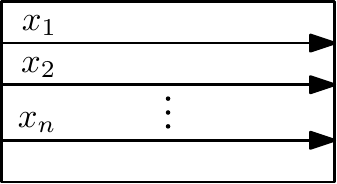}
}
\caption{Band diagram for the link $L$ in $L(p,1)$. }
\label{exx6}
\end{figure}

We have $n$ generators $x_1,\dots,x_n$ and no inner relations ($I=\emptyset$). Boundary relations are $f(x_1)=x_1$, $\dots$, $f(x_n)=x_n$, $x_nx_{n-1}\dots x_1\equiv1$ and splitting relations are $xf(x)\dots f^{p-1}(x)\equiv1$ and $f^p(x)=x$ for all $x$, so the virtual quandle is 
\begin{align}
\notag VQ(L)=\langle x_1,\dots,x_n&~|~ f(x_1)=x_1,\dots,f(x_n)=x_n,x_n\dots x_1\equiv1,\\
\notag&~~~\forall x~~xf(x)\dots f^{p-1}(x)\equiv1,\forall x~~f^p(x)=x\rangle.
\end{align}
From the relations $f(x_1)=x_1,\dots,f(x_n)=x_n$, it follows that $f(x)=x$ for all $x$. Then the virtual quandle has the following presentation
$$
\notag VQ(L)=\langle x_1,\dots,x_n~|~ f(x_1)=x_1,\dots,f(x_n)=x_n,x_n\dots x_1\equiv1, \forall x~~x^p\equiv1\rangle.$$
}
\end{example}

\begin{ttt}\label{tt2} The virtual quandle $VQ(K)$ is an invariant for links in $L(p,1)$.
\end{ttt}
\begin{proof} In order to prove that the virtual quandle is an invariant for links in  $L(p,1)$ we need to prove that it is invariant under all the Reidemeister moves $R_1$, $R_2$, $R_3$ and the $SL$-move. The proof of the invariance under the moves $R_1$, $R_2$, $R_3$ is the same as in the case of ordinary quandle for links in $S^3$ (see, for example, the proof of \cite[Theorem 3.4]{ilyman}). So we need to consider only the case of the $SL$-move.

Let $K$ be a link in lens space $L(p,1)$ whose band diagram is depicted in Figure~\ref{Ban}.
The virtual quandle $VQ(K)$ has generators $x_1,\dots,x_n,y_1,\dots,y_n,z_1,\dots,z_k$ and the set of relations $I\cup B \cup S$. Denote by $K_1$ the link having the band diagram which is obtained from the band diagram of $K$ by the $SL$-move (see Figure \ref{exx69}). Fix an orientation of $K$ and take the induced orientation on $K_1$: note that the arcs of $K_1$  labelled by $Y_1$, $\dots$, $Y_p$ are oriented  as the arc  of $K$ labelled by $y_1$.

\begin{figure}[bht!]
\noindent\centering{
\includegraphics[width=110mm]{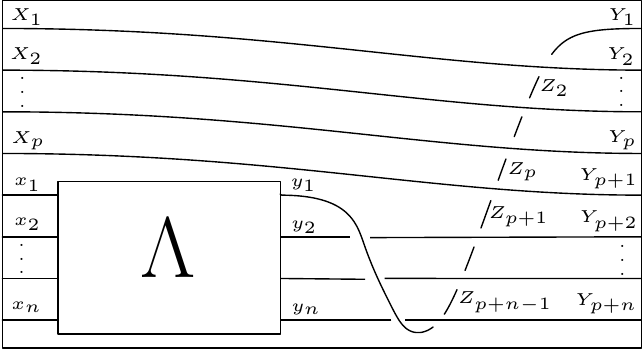}}
\caption{The band diagram after  the $SL$-move.}
\label{exx69}
\end{figure}

The virtual quandle $VQ(K_1)$ has generators $x_1$, $\dots$, $x_n$, $y_1$, $\dots$, $y_n$, $z_1$, $\dots$, $z_k$, $X_1$, $\dots$, $X_p$, $Y_1$, $\dots$, $Y_{p+n}$, $Z_2$, $\dots$, $Z_{p+n-1}$. The set of inner relations $I_1$ consists of the set $I$ and the following relations
\begin{align}
\label{eq3}X_1&=Y_2, &X_2&=Y_3, &&\dots, &X_p&=Y_{p+1},&&\\
y_2^{y_1^{\varepsilon_1}}&=Y_{p+2}, &y_3^{y_1^{\varepsilon_1}}&=Y_{p+3}, &&\dots, \label{eq4}&y_n^{y_1^{\varepsilon_1}}&=Y_{p+n},&&\\
\label{eq6}Z_{2}^{Y_{2}^{\varepsilon_{1}}}&=Y_1, &Z_{3}^{Y_{3}^{\varepsilon_{1}}}&=Z_2, &&\dots, &Z_{p+1}^{Y_{p+1}^{\varepsilon_{1}}}&=Z_p,&&\\
\label{eq5}Z_{p+2}^{Y_{p+2}^{\varepsilon_{2}}}&=Z_{p+1}, &Z_{p+3}^{Y_{p+3}^{\varepsilon_{3}}}&=Z_{p+2}, &&\dots, &Z_{p+n-2}^{Y_{p+n-1}^{\varepsilon_{n-1}}}&=Z_{p+n-1},&y_1^{Y_{p+n}^{\varepsilon_n}}&=Z_{p+n-1}.&
\end{align}
The set of boundary relations $B_1$ consists of the following $p+n$ relations
\begin{equation}\label{eq7}
f(X_1)=Y_1,~~\dots,~~f(X_p)=Y_p,~~f(x_1)=Y_{p+1},~~\dots,~~f(x_n)=Y_{p+n}
\end{equation}
and one operator relation
\begin{equation}\label{eq8}
Y_{p+n}^{\varepsilon_n}Y_{p+n-1}^{\varepsilon_{n-1}}\dots Y_{p+2}^{\varepsilon_2}Y_{p+1}^{\varepsilon_1}\dots Y_1^{\varepsilon_1} \equiv1.
\end{equation}
The set of splitting relations $S_1$ is the same as the set of splitting relations $S$ and consists of the relations $xf(x)\dots f^{p-1}(x)\equiv1$, $f^p(x)=x$ for all $x$.

From the set of relations (\ref{eq6}) and (\ref{eq5}) we see that every element $Z_i$ can be expressed via $y_1$ and $Y_1,Y_2,\dots,Y_{p+n}$. Therefore we can delete the elements $Z_2,Z_3,\dots,Z_{p+n-1}$ from the set of generators and replace the set of relations (\ref{eq6}) and (\ref{eq5}) by the relation
\begin{equation}\label{eq80}
y_1^{Y_{p+n}^{\varepsilon_n}Y_{p+n-1}^{\varepsilon_{n-1}}\dots Y_{p+2}^{\varepsilon_2}Y_{p+1}^{\varepsilon_1}\dots Y_2^{\varepsilon_2}}=Y_1.
\end{equation}
Using boundary relation (\ref{eq8}) and the fact that $Y_1^{Y_1^{\varepsilon_1}}=Y_1$ we conclude from (\ref{eq80}) that $Y_1=y_1$.

  From inner relations (\ref{eq3}) and (\ref{eq4}) we see that all the generators $Y_2,Y_3,\dots, Y_{p+n}$ can be expressed via $X_1,\dots,X_p$, $y_1,\dots,y_n$. Therefore we can delete all  these generators from the generating set changing them by their expressions via $X_1,\dots,X_p$, $y_1,\dots,y_n$ in the relations where they take part. Moreover, from the set of relations (\ref{eq7}) we can express $X_1,\dots,X_p$ via $x_1$ ($X_i=f^{p+1-i}(x_1)$) and delete all the elements $X_1,\dots,X_p$ from the set of generators.

After these manipulations we see that $VQ(K_1)$ is generated by the elements $x_1,\dots,x_n$, $y_1,\dots,y_n$, $z_1,\dots,z_k$ which are the same as for $VQ(K)$. The set of inner relations $I_1$ coincides with the set $I$. The set $B_1$ of boundary relations consists of the following relations
\begin{equation}\label{eq10}
f^{p+1}(x_1)=y_1,~~f(x_2)=y_2^{y_1^{\varepsilon_1}},~~\dots,~~f(x_n)=y_n^{y_1^{\varepsilon_1}}
\end{equation}
and one operator relation
\begin{equation}\label{eq11}
\left(y_{n}^{y_1^{\varepsilon_1}}\right)^{\varepsilon_n}\left(y_{n-1}^{y_1^{\varepsilon_1}}\right)^{\varepsilon_{n-1}}\dots\left(y_{2}^{y_1^{\varepsilon_1}}\right)^{\varepsilon_2}f(x_1)^{\varepsilon_1}f^2(x_1)^{\varepsilon_1}\dots f^{p+1}(x_1)^{\varepsilon_1}\equiv1.
\end{equation}
Using the splitting relation and the equality $y_1^{y_1^{-\varepsilon_1}}=y_1$ we can rewrite the boundary relations (\ref{eq10}) in the following way
\begin{equation}\label{eq12}
f(x_1)^{y_1^{-\varepsilon_1}}=y_1,~~f(x_2)^{y_1^{-\varepsilon_1}}=y_2,~~\dots,~~f(x_n)^{y_1^{-\varepsilon_1}}=y_n.
\end{equation}
 By Lemma  \ref{Lem1} we have $f(x_1)^{\varepsilon_1}f^2(x_1)^{\varepsilon_1}\dots f^{p}(x_1)^{\varepsilon_1}\equiv1$ and using the fact that $f^{p+1}(x_1)^{\varepsilon_1}=y_1$ we can rewrite equality (\ref{eq11}) in the form
 $\left(y_{n}^{\varepsilon_n}y_{n-1}^{\varepsilon_{n-1}}\dots y_{2}^{\varepsilon_2}y_1^{\varepsilon_1}\right)^{y_1^{\varepsilon_1}}\equiv1$,
 which is equivalent to the relation $y_{n}^{\varepsilon_n}y_{n-1}^{\varepsilon_{n-1}}\dots y_{2}^{\varepsilon_2}y_1^{\varepsilon_1}\equiv1$.

 Now the only difference between $VQ(K)$  and $VQ(K_1)$ is in the boundary relations (\ref{eq12}). However if we change $f$ by $g=i_{y_1}^{\varepsilon_1}f$, then we obtain completely the same relations in both quandles since by Lemma \ref{Lem1} the splitting relations for $f$ and for $g$ are the same.
\end{proof}

\section{Properties of the virtual quandle invariant}\label{exex}
In this section we investigate some properties of the virtual quandle  and we prove that it is an essential invariant.

At first, we recall some results on the homology of links in lens spaces 
(see  \cite[Lemma 4]{catmanmul} and \cite[Proposition 2]{gabman2}). If $K$ is  an oriented  link  in $L(p,1)$ with components $K_1,\ldots,K_m$  represented via a band diagram with $n$ right boundary points, then the homology class $[K_j]$ of $K_j$ in $H_1(L(p,1))$ has the following form
\begin{equation}\label{homj}
[K_j]=\sum_{i_j=1_j}^{n_j}\varepsilon_{i_j}~{\rm mod}~p, 
\end{equation}
where $1_j,\dots,n_j$ are the labels of the right boundary points which correspond to  $K_j$ and $\varepsilon_{i_j}=\pm 1$ according to the rule described in the previous section. Moreover, 
$$H_1(K):=H_1(L(p,1)\setminus N(K))=\mathbb Z_d \oplus_{j=1}^m \mathbb Z,$$
where $d=\gcd(p,[K_1],\ldots, [K_m])$. 

The following result connects the virtual quandle with the homology of the link.

\begin{prp}\label{prprpr} Let $K\subset L(p,1)$ be a  link  with components $K_1,\ldots,K_m$. Let $W(K)$ be the virtual quandle obtained from $VQ(K)$ by adding relations $x\equiv1$, for all  $x\in X$. Then we have
$$W(K)=\langle t_1, \dots, t_m~|~\forall j~~f^{\gcd([K_j],p)}(t_j)=t_j, \forall x~~x\equiv1\rangle$$
where $[K_j]$ denotes the homology class of $K_j$ in $H_1(L(p,1))$.  
\end{prp}
\begin{proof} 
The virtual quandle $W(K)$ is obviously a homomorphic image of $VQ(K)$ and we can easily represent $W$ by generators and relations transforming relations from $I\cup B\cup S$ using relations $x\equiv1$ for all $x$.

The boundary relation $y_n^{\varepsilon_n}y_{n-1}^{\varepsilon_{n-1}}\dots y_1^{\varepsilon_1}\equiv1$ and the splitting relation $xf(x)\dots f^{p-1}(x)\equiv1$ follow directly from the relation $x\equiv1$ for all $x$. Therefore in $W$ from the set $B$ of boundary relations we have only $f(x_1)=y_1$, $\dots$, $f(x_n)=y_n$ and from the set $S$ of splitting relations we have only the relation $f^p(x)=x$ for all $x$. Every inner relation which follows from crossings of the diagram has the form $x^y=z$ (see Figure \ref{crosrule}), therefore since  $y\equiv1$ in $W$ this relation will give us $x=z$.
It means that inner relations just identify all arcs belonging to the same component of a link and lying between two boundary points. If we denote by $t_j$ the generator corresponding to the $j$-th component of a link, we obtain the following presentation for $W(K)$.
$$W(K)=\langle t_1,\dots,t_m~|~\forall j~~f^{\varepsilon_{1_j}+\dots+\varepsilon_{m_j}}(t_j)=t_j, \forall x~~f^p(x)=x, \forall x~~x\equiv1\rangle$$
The relations $f^{\varepsilon_{1_j}+\dots+\varepsilon_{m_j}}(t_j)=t_j$ and $f^p(t_j)=t_j$ together are equivalent to the relation $f^{\gcd(\varepsilon_{1_j}+\dots+\varepsilon_{m_j},p)}(t_j)=t_j$ which by equality (\ref{homj}) is equivalent to the relation $f^{\gcd([K_j],p)}(t_j)=t_j$, therefore $W(K)$ has the following presentation
$$W(K)=\langle t_1, \dots, t_m~|~\forall j~~f^{\gcd([K_j],p)}(t_j)=t_j, \forall x~~x\equiv1\rangle$$
and the statement is proved.
\end{proof}
 
Let $\pi:S^3\to L(p,q)$ be the universal covering. An invariant of links in $L(p,q)$   is called \textit{essential} if it is able to distinguish at least a couple of links having equivalent liftings under $\pi$. The fundamental quandle for links in lens spaces is an inessential invariant since it is invariant under cyclic coverings (see \cite{fenrou}). On the contrary, for the virtual quandle invariant  we have the following result.  
\begin{prp}\label{esss} The virtual quandle is an essential invariant for links in $L(p,1)$. 
\end{prp}
\begin{proof} In order to prove that the virtual quandle is an essential invariant we need to find two links $K_1$ and $K_2$ in $L(p,1)$ which have equivalent liftings and different virtual quandles. Let us consider two links in $L(4,1)$ depicted in Figure \ref{ccc2}.
\begin{figure}[h]
\noindent\centering{
\includegraphics[width=100mm]{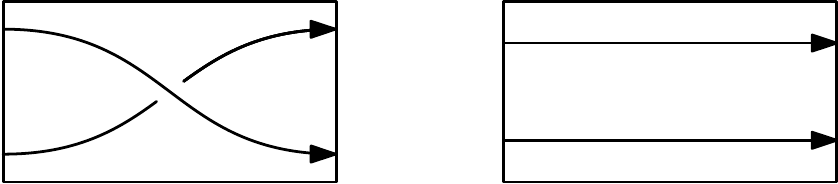}
}
\caption{Two links in $L(4,1)$ with equivalent liftings in $S^3$. }
\label{ccc2}
\end{figure}

This links have equal liftings (see \cite{manf}). By Example \ref{xxx} and Example \ref{xx} we have
\begin{align}
\notag VQ(K_1)&=\langle x~|~f^2(x)=x, xf(x)\equiv1\rangle\\
\notag VQ(K_2)&=\langle x, y~|~ f(x)=x,f(y)=y,yx\equiv1, \forall z~~z^4\equiv1\rangle
\end{align}
Let $W_1$ and $W_2$ be virtual quandles obtained from $VQ(K_1)$ and $VQ(K_2)$, respectively, adding the relations $z\equiv1$ for all $z$. Then the virtual quandles
\begin{align}
\notag W_1&=\langle x~|~f^2(x)=x, \forall z~~z\equiv1\rangle\\
\notag W_2&=\langle x, y~|~ f(x)=x, f(y)=y, \forall z~~z\equiv1\rangle
\end{align}
are not isomorphic: they both have two elements, but  in $W_1$ the automorphism $f$ permutes  the elements, while in $W_2$ the automorphism $f$ fix both of them. Then also  $VQ(K_1)$ and $VQ(K_2)$ are obviously not isomorphic.
\end{proof}

It would be interesting to investigate relationships between the virtual quandle and other link invariants: the fundamental group and the augmented  fundamental rack (see \cite{fenrou}).

~\\

\noindent Alessia CATTABRIGA\\
Department of Mathematics, University of Bologna\\
Piazza di Porta San Donato 5, 40126 Bologna, ITALY\\
e-mail: alessia.cattabriga@unibo.it

~\\
\noindent Timur NASYBULLOV\\
Department of Mathematics, KU Leuven Kulak\\
Etienne Sabbelaan 53, 8500 Kortrijk, BELGIUM\\
e-mail: timur.nasybullov@mail.ru

\begin{thebibliography}{99}
\bibitem{bakgrihed}
K.~Baker, J.~Grigsby, M.~Hedden, Grid diagrams for lens spaces
and combinatorial knot Floer homology, Int. Math. Res. Not. IMRN
10, 2008, Art. ID rnm024, 39 pp.
\bibitem{bucmau}
D.~Buck, M.~Mauricio, Connect sum of lens spaces surgeries: application
to Hin recombination, Math. Proc. Cambridge Philos. Soc., V.~150,
2011, 505-525.
\bibitem{car}
J.~Carter, A survey of quandle ideas, Introductory lectures on knot theory, 22-53, Ser. Knots Everything, 46, World Sci. Publ., Hackensack, NJ, 2012. 
\bibitem{catmanmul}
A.~Cattabriga, E.~Manfredi, M.~Mulazzani, On knots and links in lens
spaces, Topology Appl., V.~160, 2013, 430-442.
\bibitem{catmanrig}
A.~Cattabriga, E.~Manfredi, L.~Rigolli, Equivalence of two diagram representations of links in lens spaces and essential invariants, Acta Math. Hungar., V.~146, N.~1, 2015, 168-201.
\bibitem{cor}
C.~Cornwell, A polynomial invariant for links in lens spaces, J. Knot
Theory Ramifications, V.~21, 2012, 1250060, 31 pp.
\bibitem{dro}
Yu.~Drobotukhina, An analogue of the Jones polynomial for links in $\mathbb{R}P^3$ and a generalization of the Kauffman-Murasugi theorem, (Russian) Algebra i Analiz, V.~2, N.~3, 1990, 171--191; translation in Leningrad Math. J., V.~2, N.~3, 1991, 613-630
\bibitem{fenrou}
R.~Fenn, C.~Rourke, Racks and links in codimension two, J.~Knot Theory Ramifications, V.~1, N.~4, 1992, 343-406.

\bibitem{gabman}
B.~Gabrov\u{s}ek, E.~Manfredi,  On the Seifert fibered space link group, Topology Appl., V.~206, 2016, 255-275.

\bibitem{gabman2}
B.~Gabrov\u{s}ek, E.~Manfredi,  On the KBSM of links in lens spaces, ArXiv: Math/1506.01161.

\bibitem{gor} D.~V.~Gorkovets, Cocycle invariants for links in projective space, Vestn.
Chelyab. Gos. Univ. Mat. Mekh. Inform. V.~23/12, 2010, 88-97.

\bibitem{hospr}
J.~Hoste, J.~Przytycki, The $(2,\infty)$-skein module of lens spaces; a
generalization of the Jones polynomial, J.~Knot Theory Ramifications,
V.~2, 1993, 321-333.
\bibitem{ilyman}
D.~Ilyutko, V.~Manturov, Virtual knots: The state of the art, World Scientific Publishing Co, Singapore, 2013.
\bibitem{joy}
D.~Joyce, A classifying invariant of knots, the knot quandle, J.~Pure Appl. Algebra, V.~23, 1982, 37-65.
\bibitem{kau}
 L.~Kauffman, Virtual knot theory, Eur. J.~Comb., V.~20, N.~7, 1999, 663-690.
\bibitem{lamrou}
S.~Lambropoulou, C.~Rourke, Markov's theorem in 3-manifolds, Topology
Appl., V.~78, 1997, 95-122.
\bibitem{manf}
E.~Manfredi, Lift in the 3-sphere of knots and links in lens spaces, J.~Knot Theory Ramifications, V.~23, N.~5, 2014, 1450022.
\bibitem{man}
V.~Manturov, Invariants of virtual links, (Russian) Dokl. Akad. Nauk, V.~384, N.~1, 2002, 11-13, Dokl. Math, V.~65, N.~3, 329-331.
\bibitem{man2}
V.~Manturov, On invariants of virtual links, Acta Appl. Math., V.~72, N.~3, 2002,  295-309.

\bibitem{mat}
S.~Matveev, Distributive groupoids in knot theory (Russian), Mat. Sb. (N.S.), V.~119(161), N.~1(9), 1982, 78-88.

\bibitem{ste}
 S.~Stevan, Torus knots in lens spaces and topological strings, Ann. Henri Poincare, V.~16, N.~8, 2015, 1937-1967.
\end{thebibliography}
\end{document}